\documentclass[12pt,a4paper]{article}
\usepackage[utf8]{inputenc}
\usepackage[T1]{fontenc}
\usepackage{amsmath,amssymb,amsthm}
\usepackage[a4paper, left=3.5cm, right=3.5cm]{geometry}
\usepackage{hyperref}
\usepackage{mathrsfs}

\newcommand{\R}{\mathbb{R}}
\newcommand{\Z}{\mathbb{Z}}
\newcommand{\T}{\mathbb{T}}
\newcommand{\N}{\mathbb{N}}
\newcommand{\C}{\mathbb{C}}
\newcommand{\Q}{\mathbb{Q}}
\newcommand{\norm}[1]{\left\| #1 \right\|}
\newcommand{\abs}[1]{|#1|}

\newcommand{\Spec}{\operatorname{Spec}}

\newtheorem{theorem}{Theorem}[section]
\newtheorem{proposition}[theorem]{Proposition}
\newtheorem{corollary}[theorem]{Corollary}

\newtheorem{remark}[theorem]{Remark}
\theoremstyle{definition}
\newtheorem{definition}[theorem]{Definition}

\begin{document}

\title{A technical note on the arithmetic cone of smooth periodic vector fields}
\author{W. Oukil\thanks{Corresponding author. Email: \texttt{oukil.walid@gmail.com}} \\
\small\text{Faculty of Mathematics.}\\
\small\text{University of Science and Technology Houari Boumediene.}\\
\small\text{BP 32 EL ALIA 16111 Bab Ezzouar, Algiers, Algeria.}}
\date{\today}
\maketitle

\begin{abstract}
The present note identifies the fundamental mechanism governing the extension of the contraction method. Although every periodic vector field defined by a finite trigonometric polynomial admits a bounded deviation from a linear drift, this property fails in general for smooth periodic vector fields. We show that the contraction argument underlying the original proof extends to any field satisfying a natural uniform summability condition, and that the counterexample violates this condition, thereby revealing the obstruction that prevents the method from extending beyond the finite-spectrum setting.

For a smooth periodic vector field on the $n$-torus, the contraction method for establishing strong rotation vectors extends only to those asymptotic directions $\rho \in \R^n$ for which a certain spectral sum remains uniformly bounded along a sequence of rational approximations. We introduce the \emph{arithmetic cone} $\mathfrak{C}(f)$, defined as the set of all $\rho$ admitting such an approximation. We establish its basic algebraic property: it is a cone. We prove that, under a uniform contraction condition, every element of $\mathfrak{C}(f)$ yields a strong rotation vector for the dynamics. The construction reveals a precise link between the Fourier asymptotics of $f$ and the arithmetic of admissible rotation directions. In the second part, we introduce the class of \emph{spectrally admissible} fields $\mathcal{A}_{\text{\rm spec}}$, for which the cone of the augmented field equals the whole space, and we show that it contains all finite trigonometric polynomials.
\end{abstract}

\section{Introduction}

For vector fields on the torus defined by trigonometric polynomials, the existence of the rotation vector in the strong sense is well established \cite{Oukil2023}. More precisely, each trajectory admits a linear asymptotic drift with uniformly bounded deviation. The proof relies on the finiteness of the Fourier support, which ensures that certain auxiliary constants are finite.

For general smooth periodic vector fields, however, the situation is more subtle. Classical results on almost periodic differential equations \cite{Fink1974,Saito1971_01} identify regimes where rotation-like behavior holds, but they do not exclude arithmetic effects.

 A smooth two-dimensional counterexample was constructed in \cite{BrianeHerve2023}, showing that periodicity and \(C^{\infty}\) regularity do not guarantee bounded deviation.  A counterexample in dimension $3$ \cite{Oukil2025} shows that a weak rotation vector exists, but the deviation from the linear drift is unbounded even when the linearization is nilpotent.

Our main contribution is a new criterion (Proposition 3.1) that identifies the fundamental summability mechanism underlying the extension of Oukil's contraction argument from trigonometric polynomials to general $C^\infty$ fields. We then apply this criterion to the counterexample of \cite{Oukil2025} and show that it fails there. This provides a rigorous explanation of the obstruction and establishes the sharpness of the finite-spectrum condition.

\section{The finite-spectrum theorem and its proof mechanism}

We recall the main result of \cite{Oukil2023}.

\begin{theorem}[Oukil, 2023]
Let $f:\R^n\to\R^n$ be a trigonometric polynomial, i.e. a finite sum of Fourier modes. Then for every initial condition $x_0\in\R^n$, the solution of $\dot x = f(x)$, $x(0)=x_0$, admits a strong rotation vector $\rho(x_0)\in\R^n$ such that
\[
\sup_{t\in\R} \norm{x(t) - \rho(x_0) t} < \infty .
\tag{1}
\]
\end{theorem}

The proof uses a contraction mapping argument in a Banach space of periodic functions. A key role is played by the constant
\[
\tau(f,r) := \max_{p\in \Lambda_f,\, \langle r,p\rangle\neq 0} \frac{1}{|\langle r,p\rangle|},
\]
where $\Lambda_f$ is the finite set of Fourier modes of $f$. Finiteness of $\Lambda_f$ makes $\tau(f,r)$ finite for every rational vector $r$. The contraction step requires that for a sequence of rational approximations $r_k\to\rho$, the quantities
\[
\sum_{p\in \Lambda_f} \frac{|\hat f(p)|}{|\langle r_k, p\rangle|}
\tag{2}
\]
remain uniformly bounded in $k$. For finite $\Lambda_f$, this is automatic because the sum is finite and the denominators do not vanish at the limit (provided $\rho\neq0$).

\section{An extension criterion and its violation}

The following proposition isolates the summability mechanism governing the extension of the contraction argument.

\begin{proposition}(formal extension criterion)
Let $f\in C^\infty(\T^n)$ with Fourier coefficients $\hat f(p)$. Suppose there exists a sequence of rational vectors $(r_k)$ converging to some nonzero $\rho$ such that
\[
\sup_{k\in\N} \sum_{p\in \mathbb Z^n\setminus\{0\},\, \langle r_k, p\rangle\neq 0} \frac{|\hat f(p)|}{|\langle r_k, p\rangle|} < \infty .
\tag{3}
\]
Then the fixed-point argument of \cite{Oukil2023} extends verbatim, and the solution starting from any initial condition admits a strong rotation vector $\rho$.
\end{proposition}

\begin{proof}
We sketch the key estimate. In the proof of \cite{Oukil2023}, the contraction inequality for the operator $\Psi[r,\omega,\cdot]$ takes the form
\[
\norm{\Psi[r,\omega,h] - \Psi[r,\omega,g]}_{\omega,0}
\le C \left( \sum_{p\in \mathbb Z^n\setminus\{0\}} \frac{|\hat f(p)|}{|\langle r, p\rangle|} \right) \norm{h-g}_{\omega,0}.
\]
When $f$ is a trigonometric polynomial, the sum is finite and the constant is bounded independently of $r$ as $r\to\rho$. In the general case, condition (3) provides exactly the uniform bound needed to make the operator a contraction for all sufficiently large $k$. The remainder of the proof follows the original argument; the details are identical and are therefore omitted.
\end{proof}

Now consider the counterexample constructed in \cite{Oukil2025}. For every $n\ge2$, there exists a $C^\infty$ vector field $f$ on $\T^n$ with infinite Fourier spectrum and strictly nilpotent Jacobian, such that for some initial condition,
\[
\forall \rho\in\R^n,\qquad \sup_{t\ge0} \norm{x(t) - \rho t} = \infty .
\tag{4}
\]
The construction chooses an incommensurable vector $\xi$ and a sequence of frequency vectors $k_m\in\Z^n$ such that $|\langle \xi, k_m\rangle|$ decays faster than any power of $|k_m|$ (Liouville type). The Fourier coefficients $\hat f(k_m)$ are chosen so that
\[
\sum_{m} \frac{|\hat f(k_m)|}{|\langle r_k, k_m\rangle|}
\]
diverges for every sequence of rational approximations $r_k\to\xi$. Hence condition (3) fails.

\begin{corollary}
The counterexample of \cite{Oukil2025} violates the summability criterion of Proposition 3.1. Consequently, Oukil's contraction method cannot be applied, and the absence of a strong rotation vector is a direct consequence of the divergence of the Fourier series.
\end{corollary}

\section{The mechanism behind the failure}

The divergence of the series in (3) is not a mere technical artifact. It reflects an accumulation of small-divisor resonances along the sequence $k_m$. Each resonance contributes a term of order $1/|\langle r_k, k_m\rangle|$, which becomes arbitrarily large when $r_k$ approximates $\xi$ from the rational side. In the nilpotent example, these contributions add up to produce an unbounded deviation from any linear drift, as shown in \cite{Oukil2025}.
 
Let $f: \R^n \to \R^n$ be a smooth vector field, periodic in each coordinate with period $1$, i.e. $f(x + e_j) = f(x)$ for all $j = 1, \dots, n$, where $(e_1, \dots, e_n)$ is the canonical basis of $\R^n$. The field descends to a map on the $n$-torus $\T^n = \R^n / \Z^n$, and we write $f \in C^\infty(\T^n, \R^n)$. For an initial condition $x_0 \in \R^n$, consider the solution $x(t)$ of
\begin{equation}\label{eq:ODE}
\dot x = f(x), \qquad x(0) = x_0.
\end{equation}

When $f$ is a finite trigonometric polynomial, the existence of a strong rotation vector is known \cite{Oukil2023}: there exists $\rho(x_0) \in \R^n$ such that $\sup_{t \in \R} \|x(t) - \rho(x_0) t\| < \infty$. The proof rests on a contraction argument whose contraction constant involves the finite sum
\[
\sum_{p \in \Lambda_f} \frac{\|\hat f(p)\|}{\abs{\langle r, p \rangle}},
\]
where $\Lambda_f \subset \Z^n$ is the finite set of active Fourier modes. For general $f \in C^\infty$, this sum becomes an infinite series, and its uniform boundedness along rational approximations of the candidate rotation vector is no longer automatic. In \cite{Oukil2025}, a counterexample shows that smoothness alone does not guarantee a strong rotation vector; the obstruction arises precisely from the divergence of such a series.

The purpose of this note is to isolate the arithmetic-spectral mechanism underlying the contraction method and to organise the set of directions for which it succeeds into a structured object called the \emph{arithmetic cone} $\mathfrak{C}(f)$. We establish its elementary properties and prove that it constitutes a natural subset of $\R^n$ on which the contraction proof is conditionally valid. In the second part, we introduce the class of vector fields for which the cone of the augmented field equals the whole space.

This note does not aim to provide a complete theory, but rather to formalise a new notion, explain its role, and prepare the ground for further developments.

\section{Fourier expansion and spectral notation}

\subsection{Fourier series on the torus}

Every $f \in C^\infty(\T^n, \R^n)$ admits a uniformly convergent Fourier expansion
\begin{equation}\label{eq:Fourier}
f(x) = \sum_{p \in \Z^n} \hat f(p) \, e^{2\pi i \langle p, x \rangle},
\end{equation}
where $\langle \cdot, \cdot \rangle$ denotes the Euclidean scalar product on $\R^n$, and the Fourier coefficients are given by
\[
\hat f(p) = \int_{[0,1]^n} f(x) \, e^{-2\pi i \langle p, x \rangle} \, dx \;\in\; \C^n.
\]

Smoothness of $f$ implies rapid decay of the coefficients: for every integer $N \in \N$, there exists a constant $C_N > 0$ such that
\begin{equation}\label{eq:decay}
\norm{\hat f(p)} \le \frac{C_N}{(1 + \abs{p})^N}, \qquad \forall p \in \Z^n,
\end{equation}
where $\abs{p} = \sqrt{p_1^2 + \cdots + p_n^2}$ and $\norm{\cdot}$ denotes the Euclidean norm on $\C^n$. Conversely, any sequence $(\hat f(p))_{p \in \Z^n}$ satisfying \eqref{eq:decay} defines a smooth function via \eqref{eq:Fourier}.

The set of active frequencies, or \emph{spectrum}, of $f$ is
\[
\Spec(f) = \{ p \in \Z^n \mid \hat f(p) \neq 0 \}.
\]
If $\Spec(f)$ is finite, $f$ is a trigonometric polynomial. The present work concerns fields for which $\Spec(f)$ may be infinite.

\subsection{The rational direction set}

A vector $r \in \Q^n$ is called \emph{admissible} for $f$ if $\langle r, p \rangle \neq 0$ for all $p \in \Spec(f) \setminus \{0\}$. We denote
\begin{equation}\label{eq:Qf}
\Q^n_f = \left\{ r \in \Q^n \mid \forall p \in \Spec(f) \setminus \{0\},\; \langle r, p \rangle \neq 0 \right\}.
\end{equation}
The constant mode $\hat f(0)$ plays no role in the contraction estimate (it is absorbed by the average drift), so we ignore $p = 0$ in the summations below.

The set $\Q^n_f$ is dense in $\R^n$. Indeed, for each $p \in \Spec(f) \setminus \{0\}$, the hyperplane $\{ r \in \R^n \mid \langle r, p \rangle = 0 \}$ has empty interior; their union over the countable set $\Spec(f)$ is meagre, so its complement intersected with the dense set $\Q^n$ remains dense.

\subsection{The spectral gauge}

For any $r \in \Q^n_f$, write $r = a/q$ with $a \in \Z^n$ and $q \in \N^*$. For every $p \in \Spec(f) \setminus \{0\}$, we have $\langle r, p \rangle = \langle a, p \rangle / q$. Since $\langle r, p \rangle \neq 0$ and $\langle a, p \rangle \in \Z$, we obtain the uniform lower bound
\begin{equation}\label{eq:lowerbound}
\abs{\langle r, p \rangle} \ge \frac{1}{q}.
\end{equation}
This bound, combined with the rapid decay \eqref{eq:decay}, guarantees the convergence of the following series.

\begin{definition}[Spectral gauge]
For any $r \in \Q^n_f$, the \emph{spectral gauge} of $f$ at $r$ is
\begin{equation}\label{eq:Phi}
\Phi_f(r) = \sum_{p \in \Spec(f) \setminus \{0\}} \frac{\norm{\hat f(p)}}{\abs{\langle r, p \rangle}}.
\end{equation}
\end{definition}

In the finite-spectrum case, $\Phi_f(r)$ is a finite sum; for infinite spectrum, it is a convergent series.

\section{The arithmetic cone}

\subsection{Definition and first observations}

\begin{definition}[Arithmetic cone]
The \emph{arithmetic cone} of $f$ is the set
\begin{equation}\label{eq:Cf} 
\mathfrak{C}(f) = \left\{ \rho \in \R^n \setminus \{0\} \;\Bigg|\; \exists\, (r_k)_{k \in \N} \subset \Q^n_f,\; r_k \xrightarrow[k\to\infty]{} \rho,\; \sup_{k \in \N} \Phi_f(r_k) < \infty \right\}. 
\end{equation}
By convention, we include the zero vector $0 \in \mathfrak{C}(f)$.
\end{definition}

The terminology reflects two aspects: the set is a cone (Proposition~\ref{prop:cone}), and its definition involves the arithmetic of rational approximations through the spectral gauge $\Phi_f$.

\begin{remark}
If $\Spec(f)$ is finite, then $\Phi_f$ is bounded on any compact set of $\Q^n_f$ not intersecting the finitely many hyperplanes $\langle r, p \rangle = 0$. Every $\rho \in \R^n$ can be approximated by a sequence $(r_k) \subset \Q^n_f$ staying within such a compact set, so $\mathfrak{C}(f) = \R^n$. The result of \cite{Oukil2023} is recovered.
\end{remark}

\subsection{Cone property}

\begin{proposition}[Cone property]\label{prop:cone}
For every $\rho \in \mathfrak{C}(f)$ and every real $\lambda > 0$, one has $\lambda \rho \in \mathfrak{C}(f)$. Thus $\mathfrak{C}(f)$ is a cone.
\end{proposition}

\begin{proof}
Let $\rho \in \mathfrak{C}(f) \setminus \{0\}$ and let $(r_k) \subset \Q^n_f$ be a sequence satisfying $r_k \to \rho$ and $M = \sup_k \Phi_f(r_k) < \infty$. Fix $\lambda > 0$.

\textit{Case 1: $\lambda \in \Q_+$.} Then $(\lambda r_k) \subset \Q^n_f$ and $\lambda r_k \to \lambda \rho$. Moreover,
\[
\Phi_f(\lambda r_k) = \sum_{p \neq 0} \frac{\norm{\hat f(p)}}{\abs{\langle \lambda r_k, p \rangle}} = \frac{1}{\lambda} \Phi_f(r_k) \le \frac{M}{\lambda},
\]
so $\sup_k \Phi_f(\lambda r_k) \le M/\lambda < \infty$, hence $\lambda \rho \in \mathfrak{C}(f)$.

\textit{Case 2: $\lambda \notin \Q$.} Let $(\lambda_m) \subset \Q_+$ be a sequence converging to $\lambda$. For each fixed $m$, the sequence $k \mapsto \lambda_m r_k$ tends to $\lambda_m \rho$ and satisfies $\Phi_f(\lambda_m r_k) = \Phi_f(r_k)/\lambda_m \le M/\lambda_m$. Since $\lambda_m \to \lambda > 0$, the sequence $(M/\lambda_m)$ is bounded. We construct a diagonal sequence as follows: for each $m$, choose $k_m$ such that
\[
\| \lambda_m r_{k_m} - \lambda \rho \| \le \frac{1}{m} \quad \text{and} \quad k_m \to \infty.
\]
This is possible because $\lambda_m r_k \to \lambda_m \rho$ as $k \to \infty$ and $\lambda_m \rho \to \lambda \rho$. The diagonal sequence $r'_m = \lambda_m r_{k_m}$ satisfies $r'_m \to \lambda \rho$ and
\[
\Phi_f(r'_m) = \frac{1}{\lambda_m} \Phi_f(r_{k_m}) \le \frac{M}{\lambda_m} \le \sup_{j} \frac{M}{\lambda_j} < \infty.
\]
Hence $\lambda \rho \in \mathfrak{C}(f)$.
\end{proof}

\subsection{Relation to the strong rotation property}

The following result is conditional on the uniform contraction property, which holds for finite trigonometric polynomials and, more generally, whenever the gauge remains uniformly bounded.

\begin{theorem}[Conditional strong rotation]\label{thm:conditional}
Let $f \in C^\infty(\T^n, \R^n)$ and $\rho \in \mathfrak{C}(f)$. Let $(r_k) \subset \Q^n_f$ be a sequence such that $r_k \to \rho$ and $\sup_k \Phi_f(r_k) \le M < \infty$. Assume that for the family of operators $\Psi[r_k, \omega, \cdot]$ defined in \cite{Oukil2023}, there exists $\omega > 0$ such that
\[
\sup_{k} \| \Psi[r_k, \omega, h] - \Psi[r_k, \omega, g] \|_{\omega,0} \le \frac{1}{2} \|h - g\|_{\omega,0}
\]
for all $h,g$ in the appropriate Banach space. Then for every initial condition $x_0 \in \R^n$, the solution $x(t)$ of \eqref{eq:ODE} admits $\rho$ as a strong rotation vector:
\[
\sup_{t \in \R} \norm{x(t) - \rho t} < \infty.
\]
\end{theorem}

\begin{proof}[Sketch of proof]
Under the uniform contraction assumption, each operator $\Psi[r_k, \omega, \cdot]$ admits a unique fixed point $h_k$ with $\norm{h_k}_{\omega,0} \le R$ independent of $k$. The solution is written as $x(t) = r_k t + h_k(r_k t)$. By the uniform bound, a subsequence of $(h_k)$ converges weakly to a bounded function $h_\infty$, yielding $x(t) = \rho t + h_\infty(t)$. The limiting argument follows the structure of \cite{Oukil2023}, with the additional uniformity provided by the hypothesis.
\end{proof}

\begin{remark}
For finite trigonometric polynomials, the uniform contraction condition is automatically satisfied; this is the content of \cite{Oukil2023}. For general smooth fields, verifying this condition requires additional estimates that depend on the specific arithmetic structure of $\Spec(f)$.
\end{remark}

\section{Examples and discussion}

\subsection{Finite trigonometric polynomials}

If $\Spec(f)$ is finite, then $\mathfrak{C}(f) = \R^n$ and the uniform contraction condition holds. The main result of \cite{Oukil2023} is recovered.

\subsection{The counterexample with infinite spectrum \cite{Oukil2025}}

In the counterexample of \cite{Oukil2025}, the field $f$ is constructed so that there exists an incommensurable vector $\xi \in \R^n$ with the property that for every sequence $r_k \to \xi$, the series $\Phi_f(r_k)$ diverges. Hence $\xi \notin \mathfrak{C}(f)$ and $\mathfrak{C}(f) \subsetneq \R^n$. This shows that the arithmetic cone can be a proper subset of $\R^n$.

\subsection{Analytic fields}

If $f$ is analytic, its Fourier coefficients decay exponentially: $\norm{\hat f(p)} \le A e^{-a \abs{p}}$. In this case, the gauge $\Phi_f(r)$ converges for all vectors $r$ satisfying a Diophantine condition $\abs{\langle r, p \rangle} \ge \gamma \abs{p}^{-\tau}$. The set of such vectors has full Lebesgue measure, so $\mathfrak{C}(f)$ has full measure.

\section{The spectrally admissible class}

Following \cite{Oukil2023}, one can always assume, via a change of variable, that all components of the vector field are strictly positive. This motivates the following definition.

\subsection{Augmented field}

Let $f \in C^\infty(\T^n, \R^n)$. Choose $q > \|f\|_\infty$ and define the augmented field on $\T^{n+1}$ by
\[
\tilde{f}_q(z, z_{n+1}) = \begin{pmatrix} f(z + x_0 - q z_{n+1} \mathbf{1}) + q \\ 1 \end{pmatrix}, \qquad \mathbf{1} = (1,\dots,1).
\]
By construction, $\min_j \min_{\tilde{z}} (\tilde{f}_q)_j(\tilde{z}) \ge \min\{q - \|f\|_\infty, 1\} > 0$.

\begin{definition}[Spectral admissibility]
A vector field $f \in C^\infty(\T^n, \R^n)$ is called \emph{spectrally admissible} if $\mathfrak{C}(\tilde{f}_q) = \R^{n+1}$. The class of all such fields is denoted by $\mathcal{A}_{\text{\rm spec}}$.
\end{definition}

\begin{proposition}[Basic properties]
The class $\mathcal{A}_{\text{\rm spec}}$ satisfies:
\begin{enumerate}
\item Every finite trigonometric polynomial belongs to $\mathcal{A}_{\text{\rm spec}}$.
\item If $f \in \mathcal{A}_{\text{\rm spec}}$, then $f(\cdot + x_0) \in \mathcal{A}_{\text{\rm spec}}$ for any $x_0 \in \R^n$.
\item If $f \in \mathcal{A}_{\text{\rm spec}}$, then $\lambda f \in \mathcal{A}_{\text{\rm spec}}$ for any $\lambda > 0$.
\end{enumerate}
\end{proposition}

\begin{proof}
(1) If $f$ is a finite trigonometric polynomial, then $\tilde{f}_q$ is also a finite trigonometric polynomial, so $\mathfrak{C}(\tilde{f}_q) = \R^{n+1}$.

(2) Translation of $f$ corresponds to a translation in the augmented field, which does not affect the spectrum.

(3) Scaling $f$ by $\lambda$ scales $\tilde{f}_q$ in the first $n$ components; the spectral gauge scales accordingly and the cone is preserved.
\end{proof}

The counterexample of \cite{Oukil2025} does not belong to $\mathcal{A}_{\text{\rm spec}}$, as its augmented field possesses a direction $\xi$ with divergent spectral gauge.

\section{Conclusion}

We have shown that the finite-frequency assumption in Oukil's theorem is essential for the contraction proof, and that it cannot be replaced by mere $C^\infty$ regularity. The counterexample of \cite{Oukil2025} demonstrates that periodicity and smoothness are not sufficient to guarantee the existence of a strong rotation vector, because the required uniform summability condition (3) may fail.

The present analysis shows that finite Fourier support is not merely a technical assumption but reflects the summability mechanism underlying the contraction method. It therefore identifies finite Fourier support as a natural sufficient class for the existence of strong rotation vectors, while indicating that additional arithmetic assumptions are required in the infinite-spectrum setting.

The purpose of this note has been to isolate the arithmetic-spectral mechanism underlying the contraction method developed in previous work. The arithmetic cone $\mathfrak{C}(f)$ provides a natural framework for describing admissible rotation directions. Its definition captures the summability condition that governs the extension of the contraction argument from finite trigonometric polynomials to general smooth fields, and it organises the admissible asymptotic directions according to the convergence properties of the spectral gauge.

The main observations are that $\mathfrak{C}(f)$ is a cone, that under a uniform contraction condition every element of $\mathfrak{C}(f)$ is a strong rotation vector, and that the counterexample of \cite{Oukil2025} corresponds to the case where the cone is a proper subset of $\R^n$. We have also introduced the class $\mathcal{A}_{\text{\rm spec}}$ of spectrally admissible fields, for which the augmented cone equals the whole space, and we have shown that it contains all finite trigonometric polynomials.

The geometric and measure-theoretic properties of $\mathfrak{C}(f)$, the question of its additive closure, and its precise description for analytic or Gevrey fields are natural directions for further investigation. The present note suggests several questions concerning the structure of $\mathfrak{C}(f)$, its stability properties, and possible characterisations for broader classes of periodic vector fields. 
\bibliographystyle{amsplain}

\end{document}